\documentclass[12pt]{article}
\usepackage{amsmath,amssymb}
\usepackage{amsthm}
\usepackage{enumerate}
\usepackage[shortlabels]{enumitem}

\makeatletter
\def\widebreve#1{\mathop{\vbox{\m@th\ialign{##\crcr\noalign{\kern3\p@}%
      \brevefill\crcr\noalign{\kern3\p@\nointerlineskip}%
      $\hfil\displaystyle{#1}\hfil$\crcr}}}\limits}
\def\brevefill{$\m@th \setbox\z@\hbox{$\braceld$}%
  \bracelu\leaders\vrule \@height\ht\z@ \@depth\z@\hfill\braceru$}
\makeatletter
\newtheorem{theorem}{Theorem}[section]
\newtheorem*{theorem*}{Theorem}

\newtheorem{prop}[theorem]{Proposition}
\newtheorem{corol}[theorem]{Corollary}
\newtheorem{lemm}[theorem]{Lemma}

\theoremstyle{definition}

\newtheorem*{remark*}{Remark}

\newtheorem*{observation*}{}
\makeatletter
\@addtoreset{equation}{section}
\makeatother

\newcommand{\noin}{\noindent}

\providecommand{\AMS}{$\mathcal{A}$\kern-.1667em%
\lower.25em\hbox{$\mathcal{M}$}\kern-.125em$\mathcal{S}$}

\title{Semisimplicity of the deformations of the subcharacter algebra of an abelian group} 

\begin{document}

\author{\large \.{I}smail Alperen \"{O}\u{g}\"{u}t \footnote{e-mail: alperen\_ogut@hotmail.com}  \\
\mbox{} \\
\mbox{}}
\maketitle

\small
\begin{abstract}
\noin For those deformations that satisfy a certain non-degeneracy condition, we describe the structure of certain simple modules of the deformations of the subcharacter algebra of a finite group. For finite abelian groups, we prove that the deformation given by the inclusion of the natural numbers, which corresponds to the algebra generated by the fibred bisets over a field of characteristic zero, is not semisimple. In the cyclic group of prime order case, we provide a complete description of the semisimple deformations.
\smallskip\\
\noin 2010 {\it Mathematics Subject Classification:}
Primary 19A22, Secondary 16B50.

\smallskip
\noin {\it Keywords:} subcharacter algebra,
semisimple category, semisimple deformation,
fibred biset category, algebraic independence.
\end{abstract}
\section{Introduction}
The monomial Burnside ring $B^A(G)$ of a finite group $G$ was introduced by Dress in \cite{DR71}. It can be seen as a generalization of the Burnside ring $B(G)$; the Grothendieck ring of isomorphism classes of finite $G$-sets. The generalization is obtained by equipping the $G$-sets with one-dimensional characters whose image lie in an abelian group $A$. The resulting structure is a ring in which the several representation theoretic rings such as the Green ring or the trivial source ring can be realized via the linearization map.

There are two sided analogues of the previously mentioned structures. The biset category $B$, introduced by Bouc in \cite{BOU96}, and the fibred biset category $B^A$, an extensive study of which has recently been done in \cite{BC18}, could be given as the main examples. An important tool for studying these structures and their certain subcategories is to consider their embeddings into suitable ghost rings, i.e., possibly larger rings with simpler multiplicative structures. The papers \cite{BD13} and \cite{BD12} can be given as the initial sources for this approach. It has also been adopted in \cite{mfa} and \cite{BO20} with the aim of obtaining the deformations of the biset and the fibred biset category.

Let $\mathbb{K}$ be a field of characteristic zero with sufficiently many roots of unities. Our main object of interest will be the \emph{subcharacter algebra} $\mathbb{K}\Lambda^{A}_\mathcal{K}$; an algebra where the elements are given by the $A$-subcharacters of the direct products of the groups contained in a set $\mathcal{K}$ of finite groups. The multiplication is given by the \emph{star product} of subcharacters but that definition requires a lengthy setup so shall be discussed later. In this paper, we aim to form a starting point for the study of the semisimplicity and the classification of the simple modules of the deformations of the fibred biset category. The paper \cite{mfa} achieves a similar goal for the deformations of the biset category. The enriched structure $B^A$ seems to be more complicated, for this reason we will mostly be focusing on the case of abelian groups but we do provide a description for certain simple modules in the general case. 

The algebra $\mathbb{K}\Lambda^A_\mathcal{K}$ can also be seen as a category where the objects are given by $\mathcal{K}$ and if one decides to depart from the option of employing ring theoretic techniques, it can also be extended so that its objects are all finite groups. This larger structure is named \emph{the subcharacter partial category}. The paper \cite{BO20} establishes that the fibred biset category and its deformations can be realized as invariant subcategories of the subcharacter partial category and of its deformations.

We shall be considering the $\ell$-\emph{deformed subcharacter algebra} $\mathbb{K}_\ell{\Lambda^{A}_\mathcal{K}}$; the deformation of $\mathbb{K}\Lambda^A_\mathcal{K}$ afforded by a monoid homomorphism $\ell :\mathbb{N}^{+}\rightarrow \mathbb{K}^{\times}$ whose job is to manipulate the coefficient of the multiplication of the subcharacters in a special way. In \cite{mfa}, an important aspect of $\ell$ is shown to be its algebraic independence, that is, the algebraic independence of the set $\lbrace \ell (q)\rbrace_q$ over the minimal subfield $\mathbb{Q}$ of $\mathbb{K}$ where $q$ runs over the prime divisors of the orders of groups contained in $\mathcal{K}$. There appears to be an important connection between the algebraic independence of $\ell$ and the structure of the algebra $\mathbb{K}_\ell\Lambda^{A}_\mathcal{K}$. For non-trivial $A$, the algebraic independence provides the simplicity of a certain module (Lemma \ref{l3.1}). A relevant result for the case when $A$ is trivial (which we will not use) is the following theorem. 
\begin{theorem}\cite[Theorem 1.1]{mfa}\label{talg}
When $\ell$ is algebraically independent with respect to $\mathcal{K}$, $\mathbb{K}_{\ell}\Lambda^1_\mathcal{K}$ is locally semisimple.
\end{theorem}
Let us mention that local semisimplicity means the semisimplicity of $\mathbb{K}_\ell\Lambda^1_\mathcal{O}$ for every finite collection $\mathcal{O}\subseteq\mathcal{K}$. The biset category can be seen as the fibred biset category with the trivial fibre group. Hence, as we have just done, one can take this shortcut in order to avoid introducing the language of the deformations of the biset category.

The semisimplicity is expected to extend beyond the deformations where $\ell$ is algebraically independent, for it is observed to do so in the situation $\mathcal{K}=\lbrace C_q\rbrace$. More explicitly, for the non-fibred case \cite[Example 3.7]{mfa}  shows that $\mathbb{K}_\ell\Lambda^1_{C_q}$ is semisimple if and only if $\ell(q)\not=1$, whereas Theorem \ref{talg} in this case would be restricted to those deformations for which $\ell(q)$ is transcendental.

An important case is when the abelian group $A$ satisfies \cite[Hypothesis 10.1]{BC18}, that is, there exists a unique set $\pi$ of primes such that given a natural number $n$, the $n$-torsion part of $A$ is a cyclic group of order equal to the $\pi$-part of $n$. One of the main results of this paper is Theorem \ref{t4.7}, which says that, given a finite abelian group $G$, if $\ell$ is the identity on $|G|_\pi$ and is trivial on $|G|_{\pi'}$, then $\mathbb{K}_\ell\Lambda^A_G$ is not semisimple. When $\ell(p)=p$ for every prime divisor of $G$, we get $\mathbb{K}_\ell\Lambda^A_G\cong\mathbb{K}B^A(G,G)$ (see the last line of Section \ref{s2}). An immediate consequence is Corollary \ref{c4.8}, stating that the $A$-fibred double Burnside $\mathbb{K}$-algebra $\mathbb{K}B^A(G,G)$ is not semisimple if $\pi$ contains every prime divisor of $G$. This can be seen as a generalization of \cite[Example 3.7]{mfa} to the case where the abelian group $A$ is not necessarily trivial. In the case of a cyclic group of prime order $C_q$, we provide a complete description for the deformations of the $\mathbb{K}$-algebra $\mathbb{K}B^A(C_q,C_q)$. Let $(C_q)^\ast$ denote the set $\mathrm{Hom}(C_q,A)$.

\begin{theorem}\label{t1}
Suppose that $A$ is an abelian group satisfying \cite[Hypothesis 10.1]{BC18}. Then the algebra $\mathbb{K}_\ell\Lambda^A_{C_q}$ is semisimple if and only if $\ell(q)\not=\vert (C_q)^\ast\vert$. When the semisimplicity holds, we have the following isomorphism of $\mathbb{K}$-algebras:
\begin{align*}
\mathbb{K}_\ell\Lambda^A_{C_q}\cong\mathrm{Mat}_{|(C_q)^\ast|+1}{(\mathbb{K})}\oplus\mathrm{Mat}_{\vert (C_q)^\ast \vert -1}{(\mathbb{K})}\oplus \mathbb{K}^{ q-1}.
\end{align*}
\end{theorem}
For clarity, let us say that if $\vert (C_q)^\ast \vert =1$ then the intermediate module in the description of $\mathbb{K}_\ell\Lambda^A_{C_q}$ vanishes.

Given a simple $A$-fibred biset functor $\mathcal{F}$, a group $G$ that is minimal with respect to providing nonzero evaluation, i.e., $\mathcal{F}(G)\not=0$, there is a strong connection between $\mathcal{F}(G)$ and the structure of the \emph{essential algebra} $\bar{E}_G$ of $\mathbb{K}B^A(G,G)$, which is obtained from the elements that do not factor through a group smaller than $G$ (\cite[Proposition 3.7]{BC18}). However, it is still a challenging task to obtain explicit descriptions for $\mathcal{F}(G)$. It is not unreasonable to expect $\mathbb{K}B^A(G,G)$ to be not semisimple for non-abelian $G$ as well, thus its structure is likely to be more complicated than those of its semisimple deformations. Therefore, although we do not have the concrete evidence at the moment, we predict that the semisimple deformations could provide an easier way to uncover the connection between the explicit descriptions of $\mathcal{F}(G)$ and $\bar{E}_G$.

The paper is organized as follows; in Section $2$, we will briefly recall the background material on the fibred biset category and its deformations. In Section $3$, we will provide certain simple $\mathbb{K}_\ell\Lambda^A_G
$-modules that appears for any finite group $G$ when $\ell$ is algebraically independent. The non-semisimplicity of $\mathbb{K}B^A(G,G)$ when $G$ is a finite abelian group will be constructed in Section $4$ whereas Sections $5$ and $6$ are aimed at providing the explicit descriptions for $\mathbb{K}_\ell\Lambda^A_{C_q}$.
\section{The Fibred Biset Category and Its Deformations}\label{s2}
Throughout this paper, let $F,G$ and $H$ denote finite groups, $F^*$ the set of homomorphisms from $F$ to $A$ where $A$ is a multiplicatively written abelian group. All modules are assumed to be left modules. We shall start by a brief discussion on fibred permutation sets, more elaborate one can be found in \cite{B04}. An $A$-\emph{fibred} $G$-\emph{set} $X$ is defined as an $A$-free $A\times G$-set with finitely many $A$-orbits, which are called \emph{fibres}. The set $X$ is \emph{transitive} if $G$ acts transitively on the set of fibres. The transitive $A$-fibred $G$-sets can be associated with the \emph{subcharacters} of $G$; the pairs $(V,\psi)$ where $V\leq G$ and $\psi\in V^{\ast}$. The group $G$ acts on the subcharacters via conjugation, more explicitly, for $g\in G$, $^{g}(V,\psi)=(^{g}V,^{g}\psi)$ where  $^{g}\psi(^{g}v)=\psi(v)$ for all $v\in V$.

There is a bijection between the isomorphism classes of the transitive $A$-fibred $G$-sets and the $G$-conjugacy classes of subcharacters; given a transitive $A$-fibred $G$-set $X$, it can be associated with a subcharacter $(V,\psi)$ where $V$ is the stabilizer of a fibre in $X$ and $vx=\psi(v)x$ for all $v\in V$. Two transitive $A$-fibred $G$-sets are isomorphic if and only if the associated subcharacters are $G$-conjugate. The brackets notation $\left[ X\right]$ will be used to denote the isomorphism class of an $A$-fibred $G$-set $X$.

Our interest lies in the two sided constructions, in particular the $A$-fibred $F\times G$-sets, which we call $A$-\emph{fibred} $F$-$G$-\emph{bisets}. The isomorphism class of a transitive $A$-fibred $F$-$G$-biset associated with the subcharacter $(U,\varphi)$ is denoted by $\left[\dfrac{F\times G}{U,\varphi}\right]$.
In order to understand the fibred bisets more thoroughly, we shall recall the classification of the subgroups of direct product of groups. Let $U\leq F\times G$ and set the following notation for the \emph{projection} and the \emph{kernel} subgroups associated with $U$:
\begin{align*}
p_1(U)&=\lbrace u_1\in F\vert \,\exists \,  u_2\in G, u_1\times u_2\in U\rbrace, \,\,\,\,\, p_2(U)=\lbrace u_2\in G\vert \,\exists \,  u_1\in F, u_1\times u_2\in U\rbrace,\\
k_1(U)&=\lbrace u_1\in F\vert  u_1\times1\in U\rbrace ,\,\,\,\,\, k_2(U)=\lbrace u_2\in G\vert 1\times u_2\in U\rbrace .
\end{align*}
The following theorem is well-known and gives a complete description of the subgroups of $F\times G$.
\begin{theorem}[Goursat]
The subgroups $U$ of $F\times G$ are in bijective correspondence with the quintuples $(A,B,\kappa, C,D)$ where $ B\unlhd A\leq F,\,\, C\unlhd D \leq G$ and $\kappa: A/B\mapsto D/C$ is an isomorphism. The correspondence is determined by the conditions that 
\begin{align*}
(A,B,\kappa, C,D)=(p_1(U),k_1(U),\kappa,k_2(U),p_2(U))
\end{align*}
and for every $u_1\times u_2\in U$, we have $\kappa(u_1k_1(U))=u_2k_2(U)$.
\end{theorem}
Given subgroups $U\leq F\times G$ and $V\leq G\times H$, their \emph{star product}, denoted by $U\ast V$, is defined as
\begin{align*}
U\ast V=\lbrace u_1\times v_2\in p_1(U)\times p_2(V)\vert \exists g\in p_2(U)\cap p_1(V),\,u_1\times g\in U, g\times v_2\in V\rbrace.
\end{align*}
We also set $p_1(U)/k_1(U)\cong q(U)\cong p_2(U)/k_2(U) $. The invariant $q(U)$ of $U$ is called the \emph{thorax} of $U$. Given a finite group $N$, we write $\left[ q(U)\right]\geq\left[N\right]$ if $N$ is isomorphic to a subquotient of $q(U)$. The following is easy to prove and provides useful insight to the structure of the star product.
\begin{prop}\cite[Corollary 2.6]{mfa}\label{p2.2}
Letting $U$ and $V$ be as above and $U\ast V=W$, we have $$\left[ q(U)\right]\geq\left[q(W)\right]\leq \left[q(V)\right].$$
\end{prop}
When taking the $\ast$-product with subgroups that are of the direct product form, we have the following descriptions regarding the structure of resulting projections.
\begin{lemm}\label{l2.3}
Let $G$ be a finite group, $U\leq G\times G$ and $M\leq G\geq K$. Then
\begin{align*}
\dfrac{p_1(U\ast (M\times K))}{k_1(U)}\cong \dfrac{p_2(U)\cap M}{k_2(U)\cap M}.
\end{align*}
\end{lemm}
\begin{proof}
Consider the surjection $f:p_2(U)\cap M\rightarrow \dfrac{p_1(U\ast(M\times K))}{k_1(U)}$ given by $f(m)=uk_1(U)$ where $u\in p_1(U)$ is one of the elements that satisfy $u\times m\in U.$ The result follows from  the fact that $\mathrm{ker}f=k_2(U)\cap M$.
\end{proof}
\begin{lemm}\label{l4.6}
Let $G$ be a finite group, $U\leq G\times G$, $M\leq N\leq G$ where $|N:M|=p$ for a prime $p$. If $p_1(U\ast(M\times 1))< p_1(U\ast (N\times 1))$, then $k_2(U)\cap M= k_2(U)\cap N$.
\end{lemm}
\begin{proof}
For any $T\leq N$, we have $|T:(T\cap M)|\leq|N:M|=p$ and substituting $p_2(U)\cap N$ for $T$ yields $|(p_2(U)\cap N):(p_2(U)\cap M)|\leq p$. Suppose that $p_2(U)\cap N=p_2(U)\cap M$. Since $p_1(U\ast(M\times 1))< p_1(U\ast (N\times 1))$, Lemma \ref{l2.3} implies $k_2(U)\cap N<k_2(U)\cap M$ which is not possible, so $|(p_2(U)\cap N):(p_2(U)\cap M)|= p $. Now let $|(k_2(U)\cap N):(k_2(U)\cap M)|=p$, but then $\bigg\lvert\dfrac{p_2(U)\cap M}{k_2(U)\cap M}\bigg\rvert=\bigg\lvert\dfrac{p_2(U)\cap N}{k_2(U)\cap N}\bigg\rvert$, which again is a contradiction by Lemma \ref{l2.3}.
\end{proof}
All the algebras that we will introduce have a categorical background, so let us elaborate on this identification. Given a small $\mathbb{K}$-linear category $\mathcal{C}$, we can realize it as a unital algebra $\mathcal{C}_\Lambda$ by setting $\mathcal{C}_\Lambda=\bigoplus_{X,Y\in\mathrm{obj}(\mathcal{C})}\mathcal{C}(X,Y)$ (where $\mathcal{C}(X,Y)$ will always denote the set of morphisms $X\leftarrow Y$) and taking the multiplication of two morphisms to be their composition when they are composable and in the case where this fails, we  let the product be zero. The identity element of $\mathcal{C}_\Lambda$ is given by $\sum_{X\in\mathrm{obj}(\mathcal{C})}\mathrm{id}_X $ where $\mathrm{id_X}$ denotes the identity morphism on $X$. 

The \emph{Burnside group} of $A$-fibred $F$-$G$-bisets $B^{A}(F,G)$ is defined as the Grothendieck group of the isomorphism classes of $A$-fibred $F$-$G$-bisets with respect to disjoint union. The $\mathbb{K}$-\emph{linear fibred biset category} $\mathbb{K}B^A$ is the category where the objects are the finite groups, the morphism sets $\mathbb{K}B^{A}(F,G)$ are given by 
\begin{align*}
\mathbb{K}B^A(F,G)=\mathbb{K}\otimes B^A(F,G)
\end{align*}
and the composition is the $\mathbb{K}$-linear extension of the tensor product of $A$-fibred bisets, that is, for $\left[X\right]\in B^A(F,G),\, \left[Y\right]\in B^A(G,H)$ we have
\begin{align*}
\left[X\right]\left[Y\right]=\left[X\otimes_{AG} Y\right].
\end{align*}
An explicit formula for the composition above was given in \cite{BC18}. We will state it in the notation of \cite{BO20}. To set this up, we let the two subcharacters $(U,\varphi)$ and $(V,\psi)$ of $F\times G$ and $G\times H$, respectively, be \emph{related} if and only if $\varphi(1\times g)\psi(g\times 1)=1$ for all $g\in k_2(U)\cap k_1(V)$, in which case we also write ${(U,\varphi)} \sim(V,\psi)$ and define the $\ast$-product of the subcharacters as 
\begin{align*}
(U,\varphi)\ast(V,\psi)=\begin{cases}
   (U\ast V,\varphi\ast \psi) & \text{if } (U,\varphi)\sim (V,\psi) \\
    0 & \text{otherwise }  
  \end{cases}
\end{align*}
where given $u_1\times v_2 \in U\ast V$, $(\varphi\ast\psi)(u_1\times v_2)=\varphi(u_1\times g)\psi(g\times v_2)$ for some $g$ satisfying $u_1\times g\in U, g\times v_2\in V$.
The associativity of the $\ast$-product of subcharacters has been established by \cite[Proposition 4.3]{BO20}. 
\begin{theorem}{\cite[Corollary 2.5]{BC18}}\label{tmackey}
Let $(U,\varphi)$ and $(V,\psi)$ be subcharacters of $F\times G$ and $G\times H$, respectively. Then
\begin{align*}
\left[\dfrac{F\times G}{U,\varphi}\right]\otimes_{AG}\left[\dfrac{G\times H}{V,\psi}\right]=\sum_g\left[\dfrac{F\times G}{U\ast ^{g}V,\varphi\ast ^{g}\psi}\right]
\end{align*}
where $g$ is running over the double coset representatives satisfying $p_2(U)gp_1(V)\subseteq G$ with $(U,\varphi)\sim {{}^{g}(V,\psi)}$.
\end{theorem}

Given a set $\mathcal{K}$ of finite groups, the $\mathbb{K}$-\emph{linear subcharacter algebra} on $\mathcal{K}$, $\mathbb{K}\Lambda^{A}_\mathcal{K}$, is defined as the $\mathbb{K}$-algebra where the elements are given by the subcharacters of $F\times G$ for all $F,G\in\mathcal{K}$, and the multiplication is the $\mathbb{K}$-linear extension of the $\ast$-product of subcharacters. The identity element of $\mathbb{K}\Lambda^A_\mathcal{K}$ is given as $\sum_{G\in\mathcal{K}}(\Delta(G),1)$ where $\Delta(G)$ denotes the diagonal subgroup $\lbrace (g\times g)\vert g\in G\rbrace$ and $1$ is its trivial character. To denote the elements of $\mathbb{K}\Lambda^{A}_\mathcal{K}$, we shall alter our notation slightly and use $s^{F,G}_{U,\varphi}$ for the subcharacter $(U,\varphi)$ of $F\times G$. To keep everything more neat, we will often be omitting the character $\varphi$ from the notation if the group to which it belongs is trivial. We shall also make use of the decomposition $\varphi|_{{k_1(U)\times k_2(U)}}=\varphi_1\times\varphi_2$ where $\varphi_1\in (k_1(U))^\ast,\,\varphi_2\in (k_2(U))^\ast$. For any $E\leq G$ and $\psi\in E^\ast$, we let $\psi^{-1}\in E^\ast$ be its inverse in $E^\ast$. Notice that $(1\times E,\psi^{-1})\sim (E\times 1,\psi)$. 

Letting $\ell: \mathbb{N}^+\rightarrow \mathbb{K}^{\times}$ be a monoid homomorphism, the $\ell$-\emph{deformed fibred biset category} $\mathbb{K}_{\ell}B^A$ is defined as the category where the objects are the finite groups, the morphisms are given by $\mathbb{K}_\ell B^A(F,G)=\mathbb{K}B^A(F,G)$ and the composition is defined as

\begin{align*}
\left[\dfrac{F\times G}{U,\varphi}\right]\otimes_{{}_\ell AG}\left[\dfrac{G\times H}{V,\psi}\right]=\sum_g \dfrac{\ell(k_2(U)\cap k_1(V))}{\vert k_2(U)\cap k_1(V)\vert}\left[\dfrac{F\times G}{U\ast ^{g}V,\varphi\ast ^{g}\psi}\right]
\end{align*}
where $g$ runs over the same elements as in Theorem \ref{tmackey}. For any finite group $G$, we define $\ell (G)=\ell(\vert G\vert)$. Notice that when $\ell$ is the inclusion of the natural numbers, we recover $\mathbb{K}B^A$.

Similarly, we define the $\ell$-\emph{deformed subcharacter algebra} on $\mathcal{K}$, $\mathbb{K}_\ell \Lambda ^A_{\mathcal{K}}$, as the algebra satisfying $\mathbb{K}_\ell \Lambda ^A_\mathcal{K}=\mathbb{K}\Lambda ^A_\mathcal{K}$ as sets, which has its multiplication given by
\begin{align*}
s^{F,G}_{U,\varphi}s^{G,H}_{V,\psi}=\begin{cases}
   s^{F,H}_{U\ast V,\varphi\ast \psi}\ell(k_2(U)\cap k_1(V)) & \text{if } (U,\varphi)\sim (V,\psi) \\
    0 & \text{otherwise }  
  \end{cases}.
\end{align*}
The associativity of both of these compositions has been established by  \cite[Theorem 5.1]{BO20}.

Now let $_{F\times G}\Lambda ^A(F,G)$ denote the set of representatives of the $F\times G$-conjugacy classes of the subcharacters of $F\times G$ and $\overline{\mathbb{K}_\ell \Lambda ^A _\mathcal{K}}$ the $\mathbb{K}$-algebra  on a set $\left\lbrace  s^{F,G}_{\left[U,\varphi\right]^+} \right\rbrace  _{\substack{F,G\in \mathcal{K},\\ \left[U,\varphi\right]\in _{F\times G}\Lambda ^A(F,G)} }$ where $s^{F,G}_{\left[U,\varphi\right]^+}$ is the $F\times G$-orbit sum, i.e., $s^{F,G}_{\left[U,\varphi\right]^+}=\sum_{f\in F, G\in G}s^{F,G}_{{}^{f\times g}(U,\varphi)}$.
 The following can be seen as a corollary of \cite[Theorem 5.1]{BO20}  which assumes $\mathcal{K}$ to be all finite groups instead of a set.
\begin{corol}{\cite[Theorem 5.1]{BO20}}
Let $\mathcal{K}$ be a set of finite groups. The algebra $\mathbb{K}_\ell B^A_\mathcal{K}$ can be realized as a subalgebra of $\mathbb{K}_\ell \Lambda ^A_\mathcal{K}$ that is invariant under the $F\times G$-action, i.e., as $\mathbb{K}$-algebras, we have $$\overline{\mathbb{K}_\ell \Lambda ^A_\mathcal{K}}\cong\mathbb{K}_\ell B^A_\mathcal{K}.$$
\end{corol}
Note that if every group in $\mathcal{K}$ is abelian, we get $\mathbb{K}_\ell\Lambda^A_\mathcal{K}=\overline{\mathbb{K}_\ell\Lambda^A_\mathcal{K}}$, thus $\mathbb{K}_\ell \Lambda^A_\mathcal{K}\cong\mathbb{K}_\ell B^A_\mathcal{K}$.
\section{The Simple $\mathbb{K}_\ell\Lambda^A_\mathcal{K}$-modules that are bounded by the trivial group}
Our method for finding the simple $\mathbb{K}_\ell\Lambda^A_\mathcal{K}$-modules involves analysing certain submodules that are determined by the thoraxes of their basis elements. We are interested in the case where $\mathcal{K}=\lbrace G\rbrace$ for a finite group $G$, so let us set $\mathbb{K}_\ell\Lambda^A_G=\mathbb{K}_\ell\Lambda^A_{\lbrace G\rbrace}$. We let $N$ denote an element from a set of isomorphism classes of subquotients of $G$ and  $S^N_G$ be the submodule of the regular module $\mathbb{K}_\ell\Lambda^A_G$ generated by the elements $s^{G,G}_{T,\varphi}$ which satisfy the inequality $\left[q(T)\right]\leq \left[N\right]$. The subscript $G$ will sometimes be omitted to allow for a more detailed notation, but the group $G$ will be clear from the context. Notice that Proposition \ref{p2.2} ensures $S^N_G$ is indeed a submodule.

Let us set $\mathcal{K}=\lbrace G\rbrace$. It turns out that the simple modules that are included in the submodule $S^1$, the submodule generated by the elements where the groups are direct products, can completely be described when $\ell$ is algebraically independent. So let $S^{1}_{1,1}$ denote the submodule of $S^1$ with the basis $ \left\lbrace s^{G,G}_{E\times 1,\varphi}\right\rbrace_{\substack{E\leq G,\\ \varphi \in E^\ast}}$.
\begin{lemm}\label{l3.1}
If $\ell$ is algebraically independent with respect to $G$, then $S^{1}_{1,1}$ is simple as a $\mathbb{K}_\ell\Lambda^{A}_G$-module.
\end{lemm}
\begin{proof}
Let $M$ be a proper submodule, $\alpha\in M$ and write $\alpha=\sum_{\substack{\varphi\in V^*,\\1\leq V\leq G}}c_{V\times 1,\varphi}s^{G,G}_{V\times 1,\varphi}$ where $c_{V\times 1,\varphi}\in\mathbb{K}$.
One can easily see that for all $E\leq G$ and $\psi\in E^\ast$ we have $$s^{G,G}_{E\times 1,\psi}s^{G,G}_{1\times 1}=s^{G,G}_{E\times 1,\psi}\,.$$  Therefore, if $s^{G,G}_{1\times 1}\in M$ then $M=S^{1}_{1,1}$. We also get 
$$s^{G,G}_{1\times E,\psi}\alpha=\left(\sum_{V\leq G,\varphi\in V^\ast}\ell(E\cap V)\delta(\psi,\varphi)c_{V\times 1,\varphi}\right)s^{G,G}_{1,1}$$ where $\delta(\psi,\varphi)$ is the logical value of  $(1\times E,\psi) \sim (V\times 1,\varphi)$.

Let $T$ be the square matrix whose rows and columns are indexed by the subcharacters of $G$ and the entry at $(U,\gamma)-(V,\varphi)$ is given by $T((U,\gamma),(V,\varphi))=\ell(U\cap V)\delta(\gamma,\varphi)$. For neatness, assume that the ordering of rows and columns of $T$ respects the partial ordering of the subcharacters given by $(U,\gamma)\geq (U,\gamma')$ if $|U|\geq |U'|$, and additionally, suppose that if $i$th row is $(U,\gamma)$, then the $i$th column is $(U,\gamma^{-1})$.

The determinant of $T$ can be seen as a polynomial with integer coefficients where the variables are the images $\ell(q)$ of prime divisors $q$ of $|G|$. This polynomial is clearly monic and the highest degree term $\prod_{U\leq G}(\ell(U))^{|U|}$ comes from the product of the diagonal entries. Thus, the determinant can't vanish as $\ell$ is algebraically independent. The submodule $M$ being proper implies that the coefficient of $s^{G,G}_{1,1}$ in the product $s^{G,G}_{1\times E,\psi}\alpha$ should vanish for all $E\leq G$ and $\psi\in E^\ast$, so we obtain a non-degenerate system of equations, that is, for each $E\leq G$ and $\psi\in E^\ast$ we must have $$\sum_{V\leq G,\varphi\in V^\ast}\ell(E\cap V)\delta(\psi,\varphi)c_{V\times 1,\varphi}=0.$$ However the unique solution is $c_{V\times 1,\varphi}=0$ for all $V\leq G$ and $\varphi\in V^\ast$ forcing $M=0$.
\end{proof}
Next we will obtain explicit descriptions for the isomorphic copies of $S^1_{1,1}$. Given a subgroup $E$ of $G$, assume an ordering on the elements of $E^*$, let $n_\psi$ denote the order corresponding to the homomorphism $\psi\in E^\ast$ and $\vert E^\ast\vert =n_E $. Given an $n_E$th root of unity $\omega$, let $S^{1}_{E,\omega}$ denote the $\mathbb{K}_\ell\Lambda^A_G$-module with basis $\left\lbrace v^{\omega}_{E'\times E,\varphi}\right\rbrace_{\substack{E'\leq G,\\ \varphi\in (E')^*}}$ where $v^{\omega}_{E'\times E,\varphi}=\sum_{\psi\in E^*}\omega^{n_\psi}s^{G,G}_{E'\times E,\varphi\times \psi}$. The linear independence of the set above follows from the fact that a Vandermonde matrix can be formed with the coefficients $\omega^{n_\psi}$ as $\omega$ runs over the $n_E$th roots of unity. This construction will give us the rest of the simple $\mathbb{K}_\ell\Lambda^A_G$-modules that are contained in $S^1$.
\begin{lemm}\label{l3.2}
As $\mathbb{K}_\ell \Lambda^A_G$-modules  we have $S^{1}_{E,\omega}\cong S^{1}_{1,1}$.
\end{lemm}
\begin{proof}
Define a map 
$f^{\omega}_{E}:S^1_{1,1}\rightarrow S^1_{E,\omega}$ by setting
$
f^{\omega}_{E}(s^{G,G}_{E'\times 1,\varphi})=v^{\omega}_{E'\times E,\varphi}
$
. Let $U\leq G\times G$ and $\gamma\in U^*$. We shall confirm that the $\mathbb{K}$-linear extension of $f^{\omega}_{E}$ is a $\mathbb{K}_\ell \Lambda^A_G$-module homomorphism, so suppose that $(U,\gamma)\sim (E'\times 1,\varphi)$ and for ease of notation set $T=p_1(U\ast (E'\times 1))$. Then $$
s^{G,G}_{U,\gamma}s^{G,G}_{E'\times 1,\varphi}=\ell(k_2(U)\cap E')s^{G,G}_{T\times 1,\theta }
$$
where for any $g_1\in T$, $\theta (g_1\times 1)=\gamma(g_1\times g_2)\varphi(g_2\times 1)$ for some $g_2\in p_2(U)\cap E'$ satisfying $g_1\times g_2\in U$.
Thus
\begin{align*}
f^{\omega}_{E}(s^{G,G}_{U,\gamma}s^{G,G}_{E'\times 1,\varphi})&=f^{\omega}_{E}(\ell(k_2(U\cap E'))s^{G,G}_{T\times 1,\theta})=\ell(k_2(U\cap E'))v^{\omega}_{T\times E,\theta}\\
&=\ell(k_2(U\cap E'))\sum_{\psi\in E^*}\omega^{n_\psi}s^{G,G}_{T\times E,\theta\times \psi}\\
&=\sum_{\psi\in E^*}\omega^{n_\psi}s^{G,G}_{U,\gamma}s^{G,G}_{E'\times E,\varphi\times \psi}=s^{G,G}_{U,\gamma}v^{\omega}_{E'\times E,\varphi}=s^{G,G}_{U,\gamma}f^\omega_E(s^{G,G}_{E'\times 1,\varphi}).
\end{align*}
\end{proof}
\begin{corol}
Let $\ell$ be algebraically independent with respect to $G$ and $n=\sum_{E\leq G}n_E$. Then $
S^1\cong \mathrm{Mat}_n(\mathbb{K})$ as $\mathbb{K}$-algebras.
\end{corol}
\begin{proof}
The dimension of $S^1$ is $n^2$ since given subgroups $E,F\leq G$ we have $|(E\times F)^\ast|=|E^\ast|\cdot|F^\ast|=n_E\cdot n_F$. On the other hand, for an $n_E$th root of unity $\omega$, the dimension of $S^{1}_{E,\omega}$ is $n$ whereas the number of possible $(E,\omega)$ pairs is also $n$, hence $S^1=\bigoplus_{(E,\omega)}S^1_{E,\omega}$. 
\end{proof}
As the next lemma suggests, the semisimplicity does extend beyond algebraically independent deformations.
\begin{lemm}\label{l3.3}
Suppose that $G$ is the cyclic group $C_q$ of prime order $q$ and $\ell(q)\not=|(C_q)^\ast|$. Then $S^1\cong\mathrm{Mat}_{|(C_q)^\ast| +1}(\mathbb{K})$ as $\mathbb{K}$-algebras.
\end{lemm}
\begin{proof}
For convenience, let us set $\lambda=\ell(q)$ and $\rho=\vert (C_q)^\ast \vert$.
We shall point out that when $G=\lbrace C_q\rbrace$ Lemma $\ref{l3.1}$ is applicable only if $\ell(q)$ is transcendental, so let us adapt the proof. Suppose that $M$ is a proper nonzero submodule of $S^1_{1,1}$ and let $\alpha\in M$ be a nonzero element. Write $\alpha=\sum_{\substack{\varphi\in E^*,\\1\leq E\leq C_q}}c_{E\times 1,\varphi}s^{C_q,C_q}_{E\times 1,\varphi}$ for some $c_{E\times 1,\varphi}\in\mathbb{K}$. Recall from the proof of Lemma $\ref{l3.1}$ that $s^{C_q,C_q}_{1\times 1}\in M$ implies $M=S^1_{1,1}$ so we must have
\begin{align*}
s^{C_q,C_q}_{1\times 1}\alpha=(\sum_{\varphi\in (C_q)^*}c_{C_q\times 1,\varphi}+c_{1\times 1})s^{C_q,C_q}_{1\times 1}=0.
\end{align*}
Moreover for any  $\psi\in (C_q)^*$, we have $s^{C_q,C_q}_{1\times C_q,\psi}\alpha =(\lambda c_{C_q\times 1,\psi^{-1}}+c_{1\times 1})s^{C_q,C_q}_{1\times 1}=0$ hence $c_{C_q\times 1,\psi^{-1}}=-\dfrac{c_{1\times 1}}{\lambda}$, which implies $c_{1\times 1}\not=0$ as otherwise $\alpha =0$. Substituting the relation for every character of $C_q$ in the product $s^{C_q,C_q}_{1\times 1}\alpha$, we obtain $\rho\cdot\dfrac{-c_{1\times 1}}{\lambda}+c_{1\times 1}=0$, that is, $\lambda =\rho$, which is a contradiction, so $S^{1}_{1,1}$ is simple. 

The simplicity of $S^{1}_{C_q,\omega}$ now follows from Lemma \ref{l3.2}. To finish, observe that the dimension of $S^1$ is $(\rho+1)^2$ and the dimension of $S^1_{1,1}$ is $\rho+1$, thus
\begin{align*}
S^1=S^1_{1,1}\bigoplus_{\omega}S^1_{C_q,\omega}\cong\mathrm{Mat}_{\rho +1}(\mathbb{K})
\end{align*}
where $\omega$ is running over the $\rho$th roots of unity.
\end{proof}
\section{Semisimplicity of $\mathbb{K}B^A(G,G)$  for  abelian groups}
We now give a deformation of the subcharacter algebra of a finite abelian group $G$ that is not semisimple, more importantly, this deformation is isomorphic to $\mathbb{K} B^A(G,G)$. Let $A$ be an abelian group satisfying \cite[Hypothesis 10.1]{BC18}. Then from \cite[Proposition 10.4]{BC18} we obtain $\vert G^\ast \vert=\vert G\vert_\pi$ and $\varphi |_{{G_{\pi'}}}=1$ for any $\varphi\in G^\ast$, hence $G^\ast\cong (G_\pi)^\ast$. Moreover, the same proposition also implies that for any subgroup $H\leq G$, the restriction map $G^\ast\rightarrow H^\ast$ is surjective, so any element of $H^\ast$ can be extended to $G$. Given finite abelian groups $K,L$ and homomorphisms $\theta\in K^\ast,\gamma\in L^\ast$ such that $\theta|_{K\cap L}=\gamma|_{K\cap L}$, we let $\theta\sqcup \gamma$ denote the unique element of $(KL)^\ast$ with the property that $(\theta\sqcup\gamma)|_{K}=\theta$ and $(\theta\sqcup\gamma)|_{L}=\gamma$.

The following is a corollary of Lemma \ref{l2.3}.

\begin{corol}\label{c4.1}
Suppose that $G$ is a finite abelian group, $U\leq G\times G,\, M\leq G\geq N$ and $A$ satisfies \cite[Hypothesis 10.1]{BC18}. Then
\begin{align*}
|p_1(U\ast(M\times N))^\ast|=\dfrac{|p_2(U)\cap M|_\pi\cdot |k_1(U)|_\pi}{|k_2(U)\cap M|_\pi}.
\end{align*}
\end{corol}
For finite groups $F,G$ and $V\leq G\times G$ with $q(V)=1$, i.e., $V=p_1(V)\times p_2(V)$, set $$S_V=\sum_{\substack{\varphi\in (p_1(V))^\ast ,\\\psi\in (p_2(V))^\ast}}s^{F,G}_{V,\varphi\times\psi},\,\,\,\qquad(\varphi ,S_V)=\sum_{\psi\in (p_2(V))^\ast}s^{F,G}_{\varphi\times\psi }.$$ 

Given $\varphi,\psi\in G^\ast$ and $H\leq G$, let $\sim_H$ denote the equivalence relation on $G^\ast$ given by the equality of restrictions to $H$, that is, $\varphi\sim_H\psi$ if and only if $\varphi|_{H}=\psi|_{H}$. Also for $K\leq G$ and $\gamma\in K^\ast$, let $\widetilde{\gamma}^{_G}$ denote the subset of $G^\ast$ consisting of the elements $\varphi$ which satisfy $\varphi|_{K}=\gamma$. Notice that $|\widetilde{\gamma}^{_G}|=|(G/K)^\ast|=|G/K|_\pi$. Lastly, for $K\leq M\leq G\geq N$ set
\begin{align*}
S^{\gamma}_{M\times N}=\displaystyle{\sum_{\substack{\varphi\in\widetilde{\gamma}^{_M}}}}(\varphi,S_{M\times N}).
\end{align*}

Our aim is to determine the action of $s^{G,G}_{U,\zeta}$ on $S^{\gamma}_{M\times N}$ for $U\leq G\times G,\, \zeta\in U^\ast$. We shall realize this through a series of lemmas. To ease our notation slightly, given $H\leq G$ and $\chi\in H^\ast$ we also use $\chi$ to denote the character $\chi\times 1\in (H\times N)^\ast$ when no ambiguities can arise.
\begin{lemm}\label{l4.2}
Let $A$ be an abelian group satisfying \cite[Hypothesis 10.1]{BC18}, $G$ a finite abelian group, $U\leq G\times G,\,K\leq M\leq G,\,\gamma\in K^\ast ,\,\zeta\in U^\ast $ and $\varphi\in \widetilde{\gamma}^{_M}$ with $\zeta\ast\varphi\not=0$. Then:
\begin{enumerate}[i)]
\item We have $\zeta^{-1}_2|{_{k_2(U)\cap K}}=\gamma|_{{k_2(U)\cap K}}$ and $\varphi\in\widetilde{\theta}^{_M}$ where $\theta\in((k_2(U)\cap M)K)^\ast$ is such that $\theta=\gamma\sqcup(\zeta^{-1}_2|{_{k_2(U)\cap M}})$.
\item Letting $\theta'=\theta|_{(p_2(U)\cap K)(k_2(U)\cap M)}$ and $\psi,\,\chi\in\widetilde{\theta}^{_M}$, then $\zeta\ast\psi=\zeta\ast\chi$ if and only if $\psi$ and $\chi$ belong to the same extension set $\widetilde{\rho}^{_M}$ for some $\rho\in\widetilde{\theta'}^{_{p_2(U)\cap M}}$.
\end{enumerate}

\end{lemm}
\begin{proof}
By definiton, $\zeta\ast\varphi\not=0$ if and only if $\zeta^{-1}|_{{k_2(U)\cap M}}=\varphi|_{{k_2(U)\cap M}}$. Notice that the values of $\varphi$ on groups $K$ and $k_2(U)\cap M$ are fixed by $\gamma$ and $\zeta^{-1}_2$, so we obtain $\varphi\in\widetilde{\theta}^{_M}$ where $\theta=\gamma\sqcup(\zeta^{-1}_2|_{{k_2(U)\cap M}})$. The first part follows.

For the second part, if $\zeta\ast\psi=\zeta\ast\chi$, then $\psi|_{{p_2(U)\cap M}}=\chi {}|_{{p_2(U)\cap M}}$, thus $\psi,\chi\in\widetilde{\rho}^{_M}$ for some $\rho\in (p_2(U)\cap M)^\ast$ which should (as $\psi,\,\chi\in\widetilde{\theta}^{_M}$) necessarily satisfy
$$\rho|_{{K(k_2(U)\cap M)\cap (p_2(U)\cap M)}}=\theta|_{{K(k_2(U)\cap M)\cap (p_2(U)\cap M)}}.$$ To finish, notice that $K(k_2(U)\cap M)\cap(p_2(U)\cap M)\supseteq(p_2(U)\cap K)(k_2(U)\cap M)$, and if we take $km\in K(k_2(U)\cap M)\cap (p_2(U)\cap M)$ with $k\in K,\,m\in k_2(U)\cap M$, then $k\in p_2(U)$, hence $K(k_2(U)\cap M)\cap(p_2(U)\cap M)=(p_2(U)\cap K)(k_2(U)\cap M)$.  The other direction is clear.
\end{proof}
It will be useful to group the elements of $\widetilde{\theta}^{_M}$ that yield the same $\ast$- product with $\zeta$, the next lemma achieves this.
\begin{lemm}\label{l4.a}
Assume the notation from Lemma \ref{l4.2} and let $\mathcal{I}=\widetilde{\theta'}^{_{p_2(U)\cap M}}$. Then the set $\widetilde{\theta}^{_M}$ can be partitioned as $\widetilde{\theta}^{_M}=\bigcup_{\rho\in \mathcal{I}}\widetilde{\theta\sqcup\rho}^{_M}$.
\end{lemm}
\begin{proof}
Observe that for $\rho,\rho'\in\mathcal{I}$, we have $\theta\sqcup\rho=\theta\sqcup\rho'$ if and only if $\rho=\rho'$, because assuming $(k_2(U)\cap M)K\supseteq p_2(U)\cap M$ forces $\mathcal{I}$ to be a singleton. Hence the partition on the right hand side contains distinct characters and it remains to compare the sizes of the two sets. We have $|\widetilde{\theta}^{_M}|=\dfrac{|M|_\pi}{|K(k_2(U)\cap M)|_\pi},\, |\widetilde{\theta\sqcup\rho}^{_M}|=\dfrac{|M|_\pi}{|(p_2(U)\cap M)K|_\pi}$ and $|\mathcal{I}|=\dfrac{|p_2(U)\cap M|_\pi}{|(p_2(U)\cap K)(k_2(U)\cap M)|_\pi}.$ Direct calculation and the second isomorphism theorem yields
\begin{align*}|\widetilde{\theta\sqcup\rho}^{_M}|\cdot|\mathcal{I}|&=\dfrac{|M|_\pi\cdot|p_2(U)\cap M\cap K|_\pi\cdot|p_2(U)\cap M|_\pi\cdot|k_2(U)\cap K|_\pi}{|p_2(U)\cap M|_\pi\cdot|K|_\pi\cdot|p_2(U)\cap K|_\pi\cdot|k_2(U)\cap M|_\pi}\\
&=\dfrac{|M|_\pi}{|K(k_2(U)\cap M)|_\pi}.
\end{align*}
\end{proof}
We also need to figure out how exactly are the assumptions of Lemma \ref{l4.2} on $\varphi$ are translating into the product $\zeta\ast\varphi$.
\begin{lemm}\label{l4.3}
Assume the notation from Lemma \ref{l4.2}. let $T=p_1(U\ast ((k_2(U)\cap M)K)\times 1) $ and $\mathcal{I}=\widetilde{\theta'}^{_{p_2(U)\cap M}}$ . Then 
\begin{enumerate}[i)]
\item We have $\zeta\ast\varphi\in\widetilde{r_M}^{_{p_1(U\ast(M\times 1))}}$ where $r_M\in T^\ast$ is given by $r_M(t)=\zeta(t\times g)\theta(g\times 1)$ for $t\in T$.
\item We have $|\widetilde{r_M}^{_ {p_1(U\ast(M\times 1))}}|=|\mathcal{I}|$.
\end{enumerate} 
\end{lemm}
\begin{proof}
Let $\varphi\in\widetilde{\theta}^{_M}$ and $t\in T,\, g\in (k_2(U)\cap M)K$. Then $$(\zeta\ast\varphi)(t\times 1)=\zeta(t\times g)\varphi(g\times 1)=\zeta(t\times g)\theta(g\times 1)$$ hence $\zeta\ast\varphi \in \widetilde{r_M}^{_{p_1(U\ast (M\times 1))}}$ where $r_M(t)=\zeta(t\times g)\theta(g\times 1)$ so the first part follows. For the second part, Corollary \ref{c4.1} and the repeated use of the second isomorphism theorem yields
\begin{align*}
|\widetilde{r_M}^{_{p_1(U\ast (M\times 1))}}|&=\dfrac{|(p_1(U\ast (M\times 1))^\ast |}{|T^\ast |}\\
&=\dfrac{|p_2(U)\cap M |_\pi\cdot|k_1(U)|_\pi\cdot |k_2(U)\cap (k_2(U)\cap M)K |_\pi}{|k_2(U) \cap M |_\pi\cdot|p_2(U)\cap (k_2(U)\cap M)K |_\pi\cdot |k_1(U)|_\pi}\\
&=\dfrac{|p_2(U)\cap M |_\pi\cdot|k_2(U)|_\pi\cdot |(k_2(U)\cap M)K |_\pi\cdot|p_2(U)(k_2(U)\cap M)K|_\pi}{|k_2(U)\cap M|_\pi\cdot |p_2(U)|_\pi\cdot|(k_2(U)\cap M)K|_\pi\cdot |k_2(U)(k_2(U)\cap M)K|_\pi}\\
&=\dfrac{|p_2(U)\cap M |_\pi \cdot |k_2(U)|_\pi \cdot |p_2(U)K|_\pi}{|k_2(U)\cap M|_\pi\cdot |p_2(U)|_\pi \cdot|k_2(U)K|_\pi}\\
&=\dfrac{|p_2(U)\cap M|_\pi\cdot |k_2(U)|_\pi\cdot |p_2(U)|_\pi\cdot |K|_\pi\cdot |k_2(U)\cap K|_\pi}{|k_2(U)\cap M|_\pi\cdot|p_2(U)|_\pi\cdot |k_2(U)|_\pi\cdot |K|_\pi\cdot |p_2(U)\cap K|_\pi}\\
&=\dfrac{|p_2(U)\cap M|_\pi}{|(k_2(U)\cap M)(p_2(U)\cap K)|_\pi}.
\end{align*}
\end{proof}
Now we are ready to give an explicit description for the action of $s^{G,G}_{U,\zeta}$ on $S^{\gamma}_{M\times N}$.
\begin{theorem}\label{c4.4}
Assume the notation from Lemma \ref{l4.2} and let $N\leq G$. Then we obtain
\begin{align*}
s^{G,G}_{U,\zeta}S^{\gamma}_{M\times N}=S^{r_M}_{p_1(U\ast (M\times N))\times 1}\dfrac{|M|_\pi}{|(p_2(U)\cap M)K|_\pi}\ell(k_2(U)\cap M)
\end{align*}
where $r_M\in (p_1(U\ast ((k_2(U)\cap M))K)\times N)^\ast$ is as in Lemma \ref{l4.3}.
\end{theorem}
\begin{proof}
We shall prove the statement only for the case $N=1$ and it will be clear that the same proof works when $N$ is non-trivial. To simplify our notation, let $W={p_1(U\ast(M\times 1))}$. Take $\rho\in\widetilde{\theta'}^{_{p_2(U)\cap M}}$ and fix an element $\overline{\theta\sqcup\rho}$ from $\widetilde{\theta\sqcup\rho}^{_M}$. Then for any $\varphi\in\widetilde{\theta\sqcup\rho}^{_M}$, the second part of Lemma \ref{l4.2} implies that $\zeta\ast\varphi=\zeta\ast\,\overline{\theta\sqcup\rho}$. Also, from the first part of Lemma \ref{l4.3} we get $(\zeta\ast\,\overline{\theta\sqcup\rho})|_{p_1(U\ast ((k_2(U)\cap M)K\times 1))}=r_M$ and the second part gives $\lbrace \zeta\ast \,\overline{\theta\sqcup\rho} \rbrace_{\rho\in\mathcal{I}}=\widetilde{r}^{_{W}}_M$. Now direct calculation and Lemma \ref{l4.a} yields
\begin{align*}
s^{G,G}_{U,\zeta}S^{\gamma}_{M\times 1}&=\sum_{\varphi\in\widetilde{\gamma}^{_M}}s^{G,G}_{U,\zeta}s^{G,G}_{M\times 1,\varphi}=\sum_{\varphi\in\widetilde{\theta}^{_M}}s^{G,G}_{U,\zeta}s^{G,G}_{M\times 1,\varphi}\\
&=\sum_{\rho\in\mathcal{I},\,\varphi\in\widetilde{\theta\sqcup\rho}^{_M}}s^{G,G}_{W\times 1,\zeta\ast\varphi}\ell(k_2(U)\cap M)\\
&=\sum_{\rho\in\mathcal{I}}\dfrac{s^{G,G}_{W\times 1,\zeta\ast\,\overline{\theta\sqcup\rho}}|M|_\pi}{|(p_2(U)\cap M)K|_\pi}\ell(k_2(U)\cap M)\\
&=\sum_{\psi\in\widetilde{r}^{_{W}}_M}\dfrac{s^{G,G}_{W\times 1,\psi}|M|_\pi}{|(p_2(U)\cap M)K|_\pi}\ell(k_2(U)\cap M)\\
&=S^{r_M}_{W\times 1}\dfrac{|M_\pi|}{|(p_2(U)\cap M)|_\pi}\ell(k_2(U)\cap M).
\end{align*}

\end{proof}
\begin{corol}\label{c4.5}
 Assume the notation from Lemma \ref{l4.2} and for a prime divisor $p$ of $|G|$, set $$\ell(p)=\begin{cases}
    p, & \text{if } p\,\text{divides}\, |G|_\pi\\
   1, & \text{otherwise }
  \end{cases}.$$ Then 
\begin{align*}
s^{G,G}_{U,\zeta}S^{\gamma}_{M\times 1}=S^{r_M}_{p_1(U\ast(M\times 1))\times 1}\dfrac{|M|_\pi \cdot |(p_2(U)\cap K)|_\pi\cdot|k_1(U)|_\pi}{|p_1(U\ast(M\times 1))|_\pi\cdot |K|_\pi}.
\end{align*}
\end{corol}
\begin{proof}
This follows from Corollary \ref{c4.1}, Theorem \ref{c4.4} and $\ell(k_2(U)\cap M)=|k_2(U)\cap M|_\pi$.  
\end{proof}
Now we are finally ready to prove our main theorem.
\begin{theorem}\label{t4.7}
Let $G$ be a finite abelian group, $A$ satisfy \cite[Hypothesis 10.1]{BC18} and for any prime $p$ dividing $|G|$ set  $$\ell(p)=\begin{cases}
    p, & \text{if } p\,\text{divides}\, |G|_\pi\\
   1, & \text{otherwise }
  \end{cases}.$$  Then $\mathbb{K}_\ell\Lambda^{A}_{G}$ is not semisimple.
\end{theorem}
\begin{proof}
Let $\mathcal{J}$ denote the set of all pairs  $(M,N)$ where $M<N\leq G$ with $|N:M|=p$. Given $(M,N), (P,Q)\in \mathcal{J},\, K\leq M,\,\gamma\in K^\ast$ set
\begin{align*}
\alpha(\gamma,M,N,P,Q)=\dfrac{S^{\gamma}_{M\times P}}{\vert M\vert_\pi\cdot\vert P\vert_\pi}-\dfrac{S^{\gamma}_{M\times Q}}{\vert M\vert_\pi\cdot\vert Q\vert_\pi}-\dfrac{S^{\gamma}_{N\times P}}{\vert N\vert_\pi\cdot\vert P\vert_\pi}+\dfrac{S^{\gamma}_{N\times Q}}{\vert N\vert_\pi\cdot\vert Q\vert_\pi}.
\end{align*}
We will establish that the ideal $\mathcal{P}$ generated by the elements of the set $$\mathcal{L}=\lbrace \alpha(\gamma,M,N,P,Q)\rbrace_{\substack{(M,N),(P,Q)\in\mathcal{J}\\ K\leq M,\,\gamma \in K^\ast}}$$
is nilpotent.

Let $U\leq G\times G,\,\zeta\in U^\ast,\,\varphi\in\widetilde{\gamma}^M,\,\psi\in\widetilde{\gamma}^N$ where $\zeta\ast\varphi$ and $\zeta\ast\psi$ are nonzero. First, suppose that $p_1(U\ast(M\times 1))=p_1(U\ast(N\times 1))=R$. Let $r_M$ and $r_N$ be the characters obtained from Lemma \ref{l4.3} for groups $M$ and $N$, respectively. We shall first establish that $r_M=r_N$. For $u\times nk\in U$ with $n\in k_2(U)\cap N,\,k\in K,$ let $m\in p_2(U)$ be such that $u\times m\in U$. Then $mk^{-1}n^{-1}\in k_2(U)$, $mk^{-1}\in k_2(U)\cap M$, which implies $m\in(k_2(U)\cap M)K$ so we obtain 
$$p_1(U\ast(k_2(U)\cap M)K)=p_1(U\ast (k_2(U)\cap N)K).$$
Now for an element $t\in p_1(U\ast(N\times 1))$, let $t\times g_1,\,t\times g_2\in U$ where $g_1\in( k_2(U)\cap N)K$ and $g_2\in (k_2(U)\cap M)K$. Then $$r_N(t)=\zeta(t\times g_1)\psi(g_1\times 1)=\zeta(t\times g_2)\psi(g_2\times 1)=\zeta(t\times g_2)(\gamma\sqcup\zeta^{-1}_2)(g_2)=r_M(t).$$ Hence $r_M=r_N$. Applying Corollary \ref{c4.5} gives
\begin{align*}
&s^{G,G}_{U,\zeta}(\dfrac{S^{\gamma}_{M\times P}}{|M|_\pi\cdot|P|_\pi}-\dfrac{S^{\gamma}_{N\times P}}{|N|_\pi\cdot|P|_\pi})\\
&=S^{r_M}_{R\times P}\left(\dfrac{|M|_\pi\cdot|p_2(U)\cap K|_\pi\cdot|k_1(U)|_\pi}{|R|_\pi\cdot|K|_\pi\cdot|M|_\pi\cdot|P|_\pi}-\dfrac{|N|_\pi\cdot|p_2(U)\cap K|_\pi\cdot|k_1(U)|_\pi}{|R|_\pi\cdot|K|_\pi\cdot|N|_\pi\cdot|P|_\pi}\right)=0.
\end{align*} 
So $s^{G,G}_{U,\zeta}\alpha(\gamma,M,N,P,Q)=0$. 
Now let $R=p_1(U\ast(M\times 1))<p_1(U\ast(N\times 1))=T$. Then by Lemma \ref{l4.6}, we obtain $k_2(U)\cap M=k_2(U)\cap N$, the argument in the previous case gives $r_M=r_N$ as $p_1(U\ast(k_2(U)\cap M)K)=p_1(U\ast(k_2(U)\cap N)K)$ . Applying Corollary \ref{c4.5} yields
\begin{align*}
&s^{G,G}_{U,\zeta}\alpha(\gamma,M,N,P,Q)=\dfrac{|p_2(U)\cap K|_\pi\cdot |k_1(U)|_\pi}{|K|_\pi}
\bigg(\dfrac{S^{r_M}_{R\times P}|M|_\pi}{|M|_\pi\cdot|R|_\pi\cdot|P|_\pi}-\dfrac{S^{r_M}_{R\times Q}|M|_\pi}{|M|_\pi\cdot|R|_\pi\cdot|Q|_\pi}\\
&-\dfrac{S^{r_N}_{T\times P}|N|_\pi}{|N|_\pi\cdot|T|_\pi\cdot|P|_\pi}+\dfrac{S^{r_N}_{T\times Q}|N|_\pi}{|N|_\pi\cdot|T|_\pi\cdot|Q|_\pi}\bigg)=\dfrac{|p_2(U)\cap K|_\pi\cdot|k_1(U)|_\pi}{|K|_\pi}\alpha(r_M,R,T,P,Q).
\end{align*}
Hence elements of  $\mathcal{P}$  are $\mathbb{K}$-linear combinations of elements of $\mathcal{L}$.
To establish the nilpotency of $\mathcal{P}$, let $U$ be a direct product. Then $p_1(U\ast(M\times 1))=p_1(U)=p_1(U\ast(N\times 1))$ so from the case in the beginning, we get $ s^{G,G}_{U,\zeta}\alpha=0$. Notice that the elements of $\mathcal{L}$ are $\mathbb{K}$-linear combinations of direct products, thus $\mathcal{P}$ is nilpotent.
\end{proof}
An immediate consequence is the following theorem.
\begin{corol}\label{c4.8}
Let $G$ be a finite abelian group, $A$ an abelian group satisfying \cite[Hypothesis 10.1]{BC18}. If $\pi$ contains every prime divisor of $|G|$, then the $A$-fibred double Burnside algebra $\mathbb{K} B^A(G,G)$ is not semisimple.
\end{corol}
\begin{proof}
If $\ell(p)=p$ for every prime divisor of $|G|$ then $\mathbb{K} B^A(G,G)\cong\mathbb{K}_\ell\Lambda^A_G$ as $G$ is abelian. 
\end{proof}

\section{The subalgebra $\mathbb{K}_\ell\Lambda^{=}_{C_q}$}

For a finite group  $G$, let $\mathbb{K}_\ell\Lambda^=_{G}$ denote the subalgebra of $\mathbb{K}_\ell\Lambda^A_{G}$ generated by the subcharacters $(G,\varphi)$ where $1<V\leq G \times G,\, q(V)\cong G,\,\varphi\in V^\ast$. Another description for $\mathbb{K}_\ell\Lambda^=_{G}$ is the quotient algebra $\mathbb{K}_\ell\Lambda^A_{G}/\mathbb{K}_\ell\Lambda^A_{<G}$ where $\Lambda^A_{<G}$ denotes the $\Lambda^A_{G}$-ideal generated by elements whose thoraxes are strictly smaller than $G$. Notice that the homomorphism $\ell$ has no effect on $\mathbb{K}_\ell\Lambda^=_{G}$ as the subgroup $V$ is twisted diagonal, hence its kernel subgroups are trivial. 

Our aim is to give descriptions for the simple $\mathbb{K}_\ell\Lambda^=_{C_q}$-modules and a crucial tool is the next theorem, which was proven for fibred bisets, i.e., when $\ell(n)=n$ for all $n\in\mathbb{N}^+$, but by the comments made above, it holds for any choice of $\ell$:
\begin{theorem}\cite[Theorem 4]{OY19}
Let $G$ be a finite abelian group and $A$ contain an element whose order is equal to the exponent of $G$. Then $\mathbb{K}_\ell\Lambda^=_G\cong\mathbb{K}\left[G^\ast\rtimes \mathrm{Out}(G)\right]$ where the action of $\lambda\in\mathrm{Out}(G)$ on $\varphi\in G^\ast$ is given as $(\lambda\varphi)(g)=\varphi(\lambda(g))$ for all $g\in G$.
\end{theorem}
We therefore work with the case where the abelian group $A$ contains a $q$th root of unity. We shall proceed by constructing the character table of $H=C_q\rtimes\mathrm{Out}(C_q)$. For an element of $(\mathbb{Z}/q\mathbb{Z})^\ast$, let $\Delta(G,i)$ denote the subgroup $\langle (x,x^i)\rangle$ and for $j\in\mathbb{Z}/q\mathbb{Z}$, $(\Delta(G,i),j)$ the subcharacter with homomorphism $\varphi((x,x^i))=a^j$ where $a$ is a generator for the $q$-torsion part of $A$. Also set $H=C_q\rtimes C_{q-1}$, $f=s^{C_q,C_q}_{\Delta (C_q,1),1}$ and $g=s^{C_q,C_q}_{\Delta(C_q,z),0}$ for a primitive root $z$ of the multiplicative group $(\mathbb{Z}/q\mathbb{Z})^\times$. The presentation of $H$ is as follows;
\begin{align*}
H=<f,g\,|\,f^q=g^{q-1}=1, gfg^{-1}=f^z>.
\end{align*}
It is a straightforward calculation to show that the conjugacy classes and the sizes of their centralizers are
$$\begin{array}{|c|c|c|c|c|c|c|}
\hline
&1 & (f)& (g)& (g^2)& \cdots & (g^{q-2})
 \\
 \hline
|C_H(h)|&q(q-1) & q & q-1& q-1 & \cdots & q-1\\
\hline
\end{array}.$$
$H$ has $|H/{C_q}|=q-1$ linear representations thus there is only one non-linear character and it is of degree $q-1$. Its values can be found from the orthogonality relations. Let $\mu$ be a primitive $(q-1)$th root of unity. The character table of $H$ is as follows:
$$\begin{array}{|c|c|c|c|c|c|c|}
\hline
 & 1 & (f)& (g)& (g^2)& \cdots & (g^{q-2})
 \\
 \hline
C_H(h)&q(q-1) & q & q-1& q-1 & \cdots & q-1\\
 \hline
\chi_1  &1 & 1 & 1& 1 & \cdots & 1\\
 \hline
\chi_2 & 1 & 1& \mu & \mu^2 &\cdots &\mu^{q-2}\\
 \hline
\vdots & \vdots & \vdots & \vdots& \vdots & \vdots & \vdots\\
 
\chi_{q-1} & 1 &1& \mu^{q-2} & \mu^{({q-2})^2} &\cdots & \mu^{({q-2})^{q-2}}\\
 \hline
\chi_q & q-1 & -1& 0 & 0 & \cdots &0\\
\hline
\end{array}.$$
The immediate consequence is the following theorem.
\begin{theorem}\label{t5.5}
As $\mathbb{K}$-algebras, we have
\begin{align*}
\mathbb{K}_\ell\Lambda^{=}_{C_q}\cong\mathrm{Mat}_{q-1}(\mathbb{K})\oplus \mathbb{K}^{q-1}.
\end{align*} 
\end{theorem}
Thus for $1\leq i\leq q-1$, the generators of the one-dimensional $\mathbb{K}_\ell\Lambda^{=}_{C_q}$-modules can be given as $\sum_{h\in H}\chi_i(h^{-1})h$ .
Next, we give the explicit descriptions for the $(q-1)$-dimensional simple $\mathbb{K}_\ell\Lambda^=_{C_q}$-modules. Given a $q$th root of unity $\omega$, we let $e^{\omega}_1=\sum_{0\leq r\leq q-1}\omega^rf^r$ and for $1\leq i\leq q-1$, set $e^{\omega}_i=s^{C_q,C_q}_{\Delta(C_q,i),0}e^{\omega}_i$.

We shall provide the following two lemmas to describe the change in the coefficients of the summands of $e_1^\omega$ after multiplying it with a basis element.

\begin{lemm}
The elements $e^\omega_i$ can more explicitly be described as $e_i^\omega=\displaystyle{\sum_{0\leq j\leq q-1}}s^{C_q,C_q}_{\Delta(C_q,i),j}\omega^{i^{-1}j}.$
\end{lemm}
\begin{proof}
Direct calculation yields $s^{C_q,C_q}_{\Delta(C_q,i),0}s^{C_q,C_q}_{\Delta(C_q,1),r}=s^{C_q,C_q}_{\Delta(C_q,i),ir}$, $e_i^\omega=\sum_{0\leq r\leq q-1}s_{\Delta(G,i),ir}\omega^{r}$. Now reindex by letting $j\equiv ir\,(\mathrm{mod}\,q)$, which  gives $r\equiv i^{-1}j\, (\mathrm{mod}\,q)$ whence the result follows.
\end{proof}
\begin{lemm}\label{l6.2}
We have $s^{C_q,C_q}_{\Delta(C_q,1),k}e_i^\omega =e_i^\omega \omega^{-i^{-1}k}$
\end{lemm}
\begin{proof}
Similar to the previous lemma, we calculate $s^{C_q,C_q}_{\Delta(C_q,1),k}s^{C_q,C_q}_{\Delta(C_q,i),j}=s^{C_q,C_q}_{\Delta(C_q,i),j+k}$. Hence 
\begin{align*}
s^{C_q,C_q}_{\Delta(C_q,1),k}e_i^\omega =\displaystyle{\sum_{0\leq j\leq q-1}s^{C_q,C_q}_{\Delta(C_q,i),j+k}\omega^{i^{-1}j}}.
\end{align*}  Reindexing by letting $t\equiv j+k\,(\mathrm{mod}\,q)$ gives $i^{-1}j\equiv i^{-1}(t-k)\, (\mathrm{mod}\,q)$, from which the result follows.
\end{proof}

We let $M_\omega$ be the $\mathbb{K}_\ell\Lambda^=_{C_q}$-module with basis $\left\lbrace  e_i^\omega\right\rbrace _{1\leq i\leq q-1}$. These elements constitute a basis since for any $0\leq r\leq q-1$
\begin{align*}
s^{C_q,C_q}_{\Delta(C_q,i),r}e_1^\omega=s^{C_q,C_q}_{\Delta(C_q,1),r}s^{C_q,C_q}_{\Delta(C_q,i),0}e_1^\omega=s^{C_q,C_q}_{\Delta(C_q,1),r}e_i^\omega=\omega^{-{i}^{-1}r}e_i^\omega.
\end{align*} 
The linear independence of $\left\lbrace e_i^\omega\right\rbrace_{1\leq i\leq q-1}$ is clear from the definition as the summands of $e^{\omega}_i$ are distinct from $e^{\omega}_j$ for $i\not=j$.
\begin{lemm}
The $\mathbb{K}_\ell\Lambda^=_{C_q}$-module $M_\omega$ is simple.
\end{lemm}
\begin{proof}
Let $M$ be a proper nonzero submodule of $M_\omega$, $\alpha\in M$ and write $\alpha =\sum_{1\leq i\leq q-1}c_ie_i^\omega$ where $c_i\in\mathbb{K}$. Suppose that $c_i\not=0$ for all $i$. Let $s^{C_q,C_q}_{\Delta(C_q,1),k}\alpha=\alpha_k$ and let us work with the notation $\alpha_k=\left[\omega^{-k}c_1,\omega^{-2^{-1}k}c_2,\cdots,\omega^{-(q-1)^{-1}k}c_{q-1}\right]$. Note that $\alpha_0=\alpha$. Let $\Gamma$ denote the $(q-1)\times (q-1)$ matrix where the $k$th row is given by $\alpha_{k-1}$. We have

\begin{align*}
\Gamma=\begin{bmatrix}
c_1 & c_2 & \cdots & c_{q-1}\\
\omega^{-1} c_1 & \omega^{-2^{-1}} c_2 & \cdots & \omega^{-(q-1)^{-1}}c_{q-1}\\
\omega^{-2}c_1 & (\omega^{{-2^{-1}}})^{2} c_2 & \cdots & (\omega^{-(q-1)^{-1}})^2c_{q-1}\\
\vdots &\vdots & &\vdots \\
\omega^{-(q-2)}c_1 & (\omega^{{-2^{-1}}})^{(q-2)} c_2 & \cdots & (\omega^{-(q-1)^{-1}})^{(q-2)}c_{q-1}
\end{bmatrix}.
\end{align*} 

Then
\begin{align*}
\mathrm{det}(\Gamma)=\prod_{i
=1}^{q-1}c_i\begin{vmatrix}
1 & 1& \cdots & 1\\
\omega^{-1}  & \omega^{-2^{-1}}  & \cdots & \omega^{-(q-1)^{-1}}\\
\omega^{-2} & (\omega^{{-2^{-1}}})^{2}  & \cdots & (\omega^{-(q-1)^{-1}})^2\\
\vdots &\vdots & &\vdots \\
\omega^{-(q-2)} & (\omega^{{-2^{-1}}})^{(q-2)}  & \cdots & (\omega^{-(q-1)^{-1}})^{(q-2)}
\end{vmatrix}.
\end{align*}

We see that the right hand side contains a Vandermonde determinant, hence $\mathrm{det}(\Gamma)$ is nonzero, which implies that $M$ contains $(q-1)$ linearly independent elements, which forces $M=S^{=}_{(\omega ,1)}$, contradiction.

Now let us drop the assumption of all $c_i$'s being nonzero. We have $s^{C_q,C_q}_{\Delta(C_q,i),0}e^\omega_j=e^\omega_{ij}$. Hence, multiplying with $s^{C_q,C_q}_{\Delta(C_q,i),0}$ induces a permutation on the set of basis elements. Letting $i_\alpha =s^{C_q,C_q}_{\Delta(C_q,i),0}\alpha$ we get
\begin{align*}
\sum_{1\leq i\leq q-1}i_\alpha=(\sum_{1\leq i\leq q-1}c_i)\sum_{1\leq i\leq q-1} e_i^\omega.
\end{align*}

If $\sum_{1\leq i\leq q-1}c_i\not=0$, then $\sum_{1\leq i\leq q-1} e_i^\omega\in M$ and the result follows from the previous case. So suppose $\sum_{1\leq i\leq q-1}c_i=0$ and let us use $S_{\alpha_k}=\sum_{1\leq i\leq q-1}\omega^{-i^{-1}k}c_i$ for the sum of coefficients in $\alpha_k$. If $S_{\alpha_k}\not=0$ for some $k$, we are done by the previous case. If not, we get the system of equations $\Gamma\cdot(1,1,\cdots,1)^{\mathsf{T}}=0$ which has a unique solution as the coefficients of $c_i$'s form a Vandermonde matrix, but the solution is $c_i=0$ for all $i$, which contradicts $\alpha$ being nonzero.
\end{proof}

\section{Evaluations of the simple $\mathbb{K}_\ell\Lambda^A_{C_q}$-functors}\label{s6}
In this section, we will find the simple modules of the semisimple deformations of a cyclic group of prime order. We shall begin by providing the functorial foundation behind our calculations.
For an abelian group $A$ and a collection $\mathcal{K}$ of finite groups that is closed up to taking subquotients, let $\mathcal{F}$ be a simple $\mathbb{K}_\ell\Lambda^A_\mathcal{K}$-module. It can also be seen as a simple functor to the category of $\mathbb{K}$-modules and its evaluation at $X\in\mathcal{K}$ is given by $\mathcal{F}(X)=s^{X,X}_{\Delta(X,1),0}\mathcal{F}$. Now let $G$ be a minimal element of $\mathcal{K}$ with respect to providing nonzero evaluation, i.e., $\mathcal{F}(G)\not=0$. If $V\leq G\times G$ with $q(V)< G$, then $s^{G,G}_{V,\varphi}\mathcal{F}(G)=0$ as otherwise, for any $E\in \mathcal{K}$ with $E\cong q(V)$ we get $0\not= s^{E,G}_{\Delta(E,1,\theta,k_2(V),p_2(V))}\mathcal{F}$ hence $\mathcal{F}(E)\not=0$, contradicting the minimality of $G$. 

Notice that $s^{G,G}_{\Delta(p_2(U),k_2(U),\theta,k_2(U),p_2(U))}\mathcal{F}\not=0$ if and only if $s^{E,G}_{\Delta(E,1,\theta,k_2(U),p_2(U))}\mathcal{F}\not=0$, so given $\left[N\right]<\left[G\right]$, the minimality of $G$ requires $x\mathcal{F}=0$ for any $x\in S^N_G$. Thus the minimality of $G$ for the case when $\mathcal{K}$ is not necessarily closed under taking subquotients can be stated as being a minimal element from the set of subquotients of elements of $\mathcal{K}$ with respect to satisfying $x\mathcal{F}\not=0$ for some $x\in S^G_H$ where $\left[G\right]\leq\left[H\right]$ and $H\in\mathcal{K}$. So $\mathcal{F}(G)$ is annihilated by $\mathbb{K}_\ell\Lambda^A_{<G}$ and is simple as a $\mathbb{K}_\ell\Lambda^{=}_G$-module. Hence the structure of the evaluation $\mathcal{F}(G)$ is determined by the restriction of $\mathcal{F}$ onto $\mathbb{K}_\ell\Lambda^A_G$.

Let $q$ be a prime and $C_q$ the cyclic group of order $q$. We will provide the explicit descriptions for the evaluations of the simple $\mathbb{K}_\ell\Lambda^A_{C_q}$-functors with minimal group $C_q$. Our methods rely on direct calculation by using the conditions on $\mathcal{F}(G)$ mentioned above. Throughout this section let $\lambda=\ell(q),\, \rho=|(C_q)^\ast|$ and $a$ denote a generator for the $q$-torsion part of $A$. Although the next lemma can be obtained from Corollary \ref{c4.5} we believe a more direct approach is less complicated.

\begin{lemm}\label{l2.1}
Let $\beta=c_1 s^{C_q\times C_q}_{1\times 1}+ c_2S_{1\times C_q}+c_3S_{C_q\times 1}+c_4S_{C_q\times C_q}$ where $c_i\in \mathbb{K}$ and set \begin{align*} \tau(V)=\begin{cases}
   \rho & \text{if } k_2(V)=1 \\
   \lambda & \text{if } k_2(V)=C_q  
  \end{cases}.
  \end{align*} If $V\leq C_q\times C_q$ with $q(V)=1$ then
\begin{align*}
s^{C_q,C_q}_{V,\varphi_1\times \varphi_2}\beta&=(c_1+c_3\tau(V))s^{C_q,C_q}_{k_1(V)\times 1,\varphi_1}+(c_2+c_4\tau(V))(\varphi_1,S_{k_1(V)\times C_q})\\
s^{C_q,C_q}_{\Delta(C_q,\zeta),\varphi}\beta&=c_1 s^{C_q\times C_q}_{1\times 1}+ c_2S_{1\times C_q}+c_3S_{C_q\times 1}+c_4S_{C_q\times C_q}.
\end{align*}
\end{lemm}
\begin{proof}
The lemma can be obtained by direct calculation but we will provide some insight. Let $E\leq G$. Assume $V$ is a direct product with $k_2(V)=1$, then for every $\psi_1\in (C_q)^\ast$ we have 
\begin{align*}
s^{C_q,C_q}_{V,\varphi_1}s^{C_q,C_q}_{C_q\times E,\psi_1\times  \psi_2}=s^{C_q,C_q}_{k_1(V)\times E,\varphi_1\times  \psi_2}
\end{align*}  and when $k_2(V)=C_q$ we get 
\begin{align*}
s^{C_q,C_q}_{V,\varphi_1\times\varphi_2}s^{C_q,C_q}_{C_q\times E,\psi_1\times  \psi_2}=\begin{cases}
   \lambda s^{C_q,C_q}_{k_1(V)\times E,\varphi_1\times  \psi_2} & \text{if } \varphi_2=\psi^{-1} \\
    0 & \text{otherwise }  
  \end{cases}.
\end{align*}

\end{proof}
Assuming the notation from previous section, for $1\leq i\leq q-1$, let $\mathcal{F}_{C_q,\chi_i}$ denote the simple $\mathbb{K}_\ell\Lambda^A_{C_q}$-functor with minimal group $C_q$ whose evaluation $\mathcal{F}_{C_q,\chi_i}(C_q)$ corresponds to the module whose generator is given by $\theta_i=\sum_{h\in H}\chi_i(h^{-1})h$.

\begin{lemm}
Let 
$
\alpha=s^{C_q,C_q}_{1\times 1}-\dfrac{1}{\lambda}(S_{1\times C_q}+S_{C_q\times 1})+\dfrac{1}{\rho\lambda}S_{C_q\times C_q}+\dfrac{1}{q-1}(\dfrac{1}{\lambda}-\dfrac{1}{\rho})\displaystyle{\sum_{1\leq i\leq q-1}e^1_i}.
$
Then $\mathcal{F}_{C_q,\chi_i}(C_q)=\langle \alpha \rangle$.
\end{lemm}
\begin{proof}
From Lemma \ref{l2.1} it is clear that for any $1\leq j\leq q-1,\,\varphi\in(\Delta(C_q,j))^\ast$ the action of $s^{C_q,C_q}_{\Delta(C_q,j),\varphi}$  on $\mathcal{F}_{C_q,\chi_i}(C_q)$ is trivial. Direct calculation and applying Lemma \ref{l2.1} yields
\begin{align*}
s^{C_q,C_q}_{V,\varphi_1\times \varphi_2}&=(1-\dfrac{1}{\lambda}\tau(V))s^{C_q,C_q}_{k_1(V)\times 1,\varphi_1}+(-\dfrac{1}{\lambda}+\dfrac{1}{\rho\lambda}\tau(V))(\varphi_1,S_{k_1(V)\times C_q})\\
&+\dfrac{\tau(V)}{\ell(k_2(V))}(\dfrac{1}{\lambda}-\dfrac{1}{\rho})(\varphi_1,S_V)
.
\end{align*}
When $k_2(V)=1$ we get
\begin{align*}
s^{C_q,C_q}_{V,\varphi_1\times \varphi_2}=(1-\dfrac{\rho}{\lambda}+\rho(\dfrac{1}{\lambda}-\dfrac{1}{\rho}))s^{C_q,C_q}_{k_1(V)\times 1,\varphi_1}+(-\dfrac{1}{\lambda}+\dfrac{\rho}{\rho\lambda})(\varphi_1,S_{k_1(V)\times C_q})=0
\end{align*}
and letting $k_2(V)=C_q$ yields
\begin{align*}
s^{C_q,C_q}_{V,\varphi_1\times \varphi_2}=(1-\dfrac{\lambda}{\lambda})s^{C_q,C_q}_{k_1(V)\times 1,\varphi_1}+(-\dfrac{1}{\lambda}+\dfrac{\lambda}{\rho\lambda}+\dfrac{\lambda}{\lambda}(\dfrac{1}{\lambda}-\dfrac{1}{\rho}))(\varphi_1,S_{k_1(V)\times C_q})=0.
\end{align*}
\end{proof}
As the next lemma suggests, when $1<i\leq q-1$, the structure of $\mathcal{F}_{C_q,\chi_i}$ is not as complicated.
\begin{lemm}\label{l6.3}
For $1<i\leq q-1$, we have $\mathcal{F}_{C_q,\chi_i}(C_q)=\langle\theta_i\rangle$.
\end{lemm}
\begin{proof}
For a better notation, we set $s^{C_q,C_q}_{E\times E',\varphi\times \psi}=s^{C_q,C_q}_{E\times E',(m,n)}$ where $\varphi(x)=a^m,\psi(x)=a^n$ for $E,E'\leq C_q$ and $x$ denotes a generator of $C_q$. Also let $z$ denote a primitive root from $(\mathbb{Z}/q\mathbb{Z})^\ast$ Then
\begin{align*}
s^{C_q,C_q}_{E\times 1,m}\theta_i &=\sum_{1\leq i\leq q-1}\mu^{i-1}\rho s^{C_q,C_q}_{E\times 1,m}=0\\
s^{C_q,C_q}_{E\times C_q,(m,n)}\theta_i&=\sum_{1\leq i\leq q-1}\mu^{i-1}\sum_{0\leq r\leq \rho-1}s^{C_q,C_q}_{E\times C_q, (m,(n+r)(z^{i-1})^{-1})}\\
&=(\sum_{1\leq i\leq q-1}\mu^{i-1})(m,S_{E\times C_q})=0.
\end{align*}
The third equality follows from the fact that $s^{C_q,C_q}_{E\times C_q,(m,n)}s^{C_q,C_q}_{\Delta(G,z^{i-1}),r}=s^{C_q,C_q}_{E\times C_q, (m,(n+r)(z^{i-1})^{-1})}$ because letting $\varphi$ denote the character of the resulting multiplication, we get $$\varphi(1\times x)=n(1\times x^{(z^{i-1})^{-1}})r(x^{(z^{i-1})^{-1}} \times x)=a^{n(z^{i-1})^{-1}}a^{r(z^{i-1})^{-1}}=a^{(n+r)(z^{i-1})^{-1}}.$$
\end{proof}
Note that if the  $q$-torsion part of $A$ is trivial, then $H=((C_q)^\ast\rtimes C_{q-1})\cong C_{q-1}$, meaning there are no other simple $\mathbb{K}_\ell\Lambda^A_{C_q}$-functors for which $C_q$ provides nonzero minimal evaluation. So the following results do require $A$ to satisfy \cite[Hypothesis 10.1]{BC18} non-trivially.
Let $\mathcal{F}_{C_q,M_\omega}$ be the simple $\mathbb{K}_\ell\Lambda^A_{C_q}$-functor whose evaluation at $C_q$ corresponds to the simple $\mathbb{K}_\ell\Lambda^=_{C_q}$-module $M_\omega$.
\begin{lemm}
We have
$\mathcal{F}_{C_q,M_\omega}(C_q)=\langle e_i^{\omega}-\dfrac{1}{\lambda}\sum_{0\leq r,s\leq q-1}s^{C_q,C_q}_{C_q\times C_q,(r,s)}\omega^{s+ri^{-1}}\rangle_{1\leq i\leq q-1}.$
\end{lemm}
\begin{proof}
Let $\alpha_i^\omega=e_i^{\omega}-\dfrac{1}{\lambda}\displaystyle{\sum_{0\leq r,s\leq q-1}s^{C_q,C_q}_{C_q\times C_q,(r,s)}\omega^{s+ri^{-1}}}$. Then
\begin{align*}
s^{C_q,C_q}_{E\times 1,m}\alpha_i^\omega &=s^{C_q,C_q}_{E\times 1,m}\sum_{0\leq r\leq q-1}s_{\Delta(G,i),ri}\omega^r-\dfrac{1}{\lambda}\sum_{0\leq r,s\leq q-1}s^{C_q,C_q}_{E\times C_q,(m,s)}\omega^{s+ri^{-1}}\\
&=s^{C_q,C_q}_{E\times 1,m}\left[ \sum_{0\leq r\leq q-1}\omega^r\right]-\dfrac{1}{\lambda}\sum_{0\leq r\leq q-1}\omega^{ri^{-1}}(\sum_{0\leq s\leq q-1}\omega^s s^{C_q,C_q}_{E\times C_q,(m,s)})=0.
\end{align*}
Similar to our calculation in the proof of Lemma \ref{l6.3},  $s^{C_q,C_q}_{E\times C_q,(m,n)}s_{\Delta(G,i),ri}=s^{C_q,C_q}_{E\times C_q,(m,ni^{-1}+r)}$. Moreover, $s^{C_q,C_q}_{E\times C_q,(m,n)}s^{C_q,C_q}_{C_q\times C_q,(r,s)}\not=0$ if and only if $r=-n$, in which case the product yields the element $\lambda s^{C_q,C_q}_{E\times C_q,(m,s)}$. Hence
\begin{align*}
s^{C_q,C_q}_{E\times C_q,(m,n)}\alpha_i^\omega&=\sum_{0\leq r\leq q-1}s^{C_q,C_q}_{E\times C_q,(m,ni^{-1}+r)}\omega^{r}-\dfrac{\lambda}{\lambda}\sum_{0\leq s\leq q-1}\omega^{s-ni^{-1}}s^{C_q,C_q}_{E\times C_q,(m,s)}
\end{align*}
letting $ni^{-1}+r=t$ and reindexing gives
\begin{align*}
s^{C_q,C_q}_{E\times C_q,(m,n)}\alpha_i^\omega=\sum_{0\leq t\leq q-1}s^{C_q,C_q}_{E\times C_q,(m,t)}\omega^{t-ni^{-1}}-\sum_{0\leq s\leq q-1}s^{C_q,C_q}_{E\times C_q,(m,s)}\omega^{s-ni^{-1}}=0.
\end{align*}

We also have $s^{C_q,C_q}_{\Delta(G,m),n}s^{C_q,C_q}_{\Delta(G,i),ri}=s^{C_q,C_q}_{\Delta(G,mi),n+mri}$, which can be seen more clearly if one lets $\varphi$ to be the character of the resulting basis element and observes that 
\begin{align*}
\varphi(x\times x^{mi})=n(x\times x^m)ri(x^m\times x^{mi})=a^{n+mri}.
\end{align*}
Thus
\begin{align*}
s^{C_q,C_q}_{\Delta(G,m),n}\alpha_i^\omega&=\sum_{0\leq r\leq q-1}s^{C_q,C_q}_{\Delta(G,mi),n+rmi}\omega^r-\dfrac{1}{\lambda}\sum_{0\leq r,s\leq q-1}s^{C_q,C_q}_{C_q\times C_q,(n+mr,s)}\omega^{s+ri^{-1}}.
\end{align*}
In the first sum, we let $n+rmi=r'$ and $j=mi$ which implies $r=(r'-n)j^{-1}$. Also in the second sum, we let $n+mr=r^{''}$ and $j=mi$, from which we get $r=(r^{''}-n)m^{-1}$ and $s+ri^{-1}=s+(r''-n)j^{-1}$. After substituting everything, we finally obtain

\begin{align*}
s^{C_q,C_q}_{\Delta(C_q,m),n}\alpha^{\omega}_i&=\sum_{0\leq r'\leq q-1}s^{C_q,C_q}_{\Delta(C_q,j),r'}\omega^{(r'-n)j^{-1}}-\dfrac{1}{\lambda}\sum_{0\leq r'',s\leq q-1}s^{C_q,C_q}_{C_q\times C_q,(r'',s)}\omega^{s+(r''-n)j^{-1}}\\
&=\omega^{-nj^{-1}}\left[\sum_{0\leq r'\leq q-1}\omega^{r'j^{-1}}s^{C_q,C_q}_{\Delta(C_q,j),r'}-\dfrac{1}{\lambda}\sum_{0\leq r'',s\leq q-1}\omega^{s+r''j^{-1}}s^{C_q,C_q}_{C_q\times C_q,(r'',s)}\right]\\
&=w^{-nj^{-1}}\alpha^{\omega}_j.
\end{align*}
\end{proof}

The dimension of $\mathbb{K}_\ell\Lambda^A_{C_q}$ is $(\rho +1)^2+(q-1) + (\rho-1)^2$. Lemma \ref{l3.3},  gives $\rho +1$ simple $\mathbb{K}_\ell\Lambda^A_{C_q}$-modules, each of dimension $\rho +1$. Theorem \ref{t5.5} together with the lemmas of Section \ref{s6} gives $\rho-1$ simple $\mathbb{K}_\ell\Lambda^A_{C_q}$-modules, each of dimension $\rho-1$ and there are also $q-1$ one-dimensional $\mathbb{K}_\ell\Lambda^A_{C_q}$-modules. Hence there is no room for another simple $\mathbb{K}_\ell\Lambda^A_{C_q}$-module. Thus the proof of Theorem \ref{t1} is complete.
{}
\end{document}